\documentclass[11pt]{article}
\usepackage{amsfonts, amsmath, amsthm, amssymb, mathtools, enumitem, tikz, hyperref, caption}
\usepackage{eucal}
\usepackage[noabbrev,capitalize]{cleveref}

\definecolor{darkgray}{RGB}{64,64,64}
\definecolor{litegray}{RGB}{192,192,192}
\definecolor{green}{RGB}{34,116,66}
\definecolor{red}{RGB}{232,102,116}
\tikzstyle{vertex}=[circle, draw=white, fill=red, inner sep=0pt, minimum width=5pt]
\tikzstyle{vtx}=[circle, draw, fill=litegray, inner sep=0pt, minimum width=5pt]

\usepackage{microtype, fullpage}
\DeclarePairedDelimiter\abs{\lvert}{\rvert}%
\DeclarePairedDelimiter\ceil{\lceil}{\rceil}%
\DeclarePairedDelimiter\floor{\lfloor}{\rfloor}%

\newtheorem{theorem}{Theorem}[section]
\newtheorem{lemma}[theorem]{Lemma}
\newtheorem{proposition}[theorem]{Proposition}
\newtheorem{corollary}[theorem]{Corollary}

\theoremstyle{definition}
\newtheorem{definition}[theorem]{Definition}

\theoremstyle{remark}
\newtheorem*{remark}{Remark}
\newtheorem*{claim*}{Claim}

\newcommand{\EE}{\mathcal{E}}
\newcommand{\PP}{\mathcal{P}}
\newcommand{\PPP}{\bar{\PP}}
\newcommand{\OO}{\mathcal{O}}
\newcommand{\KK}{\mathcal{K}}
\newcommand{\R}{\mathbb{R}}
\newcommand{\F}{\mathbb{F}}
\newcommand{\union}{\cup\,} 
\newcommand{\set}[1]{\left\{#1\right\}}
\newcommand{\dset}[2]{\set{#1\colon #2}}
\newcommand{\rank}[1]{\mathrm{rank}(#1)}

\title{Rainbow odd cycles}

\author{
  Ron Aharoni\footnote{Department of Mathematics, Technion -- Israel Institute of Technology, Technion City, Haifa 3200003, Israel.}${\ }^{,}$\thanks{Email: {\tt ra@tx.technion.ac.il}. Supported in part by the United States--Israel Binational Science Foundation grant no.\ 2006099, the Israel Science Foundation grant no.\ 2023464 and the Discount Bank Chair at the Technion. This paper is part of a project that has received funding from the European Union's Horizon 2020 research and innovation programme under the Marie Sk{\l{}}odowska-Curie grant agreement no.\ 823748.}
  \and
  Joseph Briggs\footnotemark[1]${\ }^{,}$\thanks{Email: {\tt briggs@campus.technion.ac.il}.}
  \and
  Ron Holzman\footnotemark[1]${\ }^{,}$\thanks{Email: {\tt holzman@technion.ac.il}. Research partly done during a visit at the Department of Mathematics, Princeton University, supported by the H2020-MSCA-RISE project CoSP--GA no.\ 823748.}
  \and
  Zilin Jiang\thanks{School of Mathematical and Statistical Sciences, and
  School of Computing and Augmented Intelligence, Arizona State University, Tempe, AZ 85281, USA. Email: {\tt zilinj@asu.edu}. The work was done when Z.~Jiang was an Applied Mathematics Instructor at Massachusetts Institute of Technology, and was supported in part by an AMS--Simons Travel Grant, and by U.S. taxpayers through the NSF grant DMS-1953946.}
}

\date{}

\begin{document}

\maketitle

\begin{abstract}
  We prove that every family of (not necessarily distinct) odd cycles $O_1, \dots, O_{2\lceil n/2 \rceil-1}$ in the complete graph $K_n$ on $n$ vertices has a rainbow odd cycle (that is, a set of edges from distinct $O_i$'s, forming an odd cycle). As part of the proof, we characterize those families of $n$ odd cycles in $K_{n+1}$ that do not have any rainbow odd cycle. We also characterize those families of $n$ cycles in $K_{n+1}$, as well as those of $n$ edge-disjoint nonempty subgraphs of $K_{n+1}$, without any rainbow cycle.
\end{abstract}

\section{Introduction}\label{sec:introduction}

Given a family $\EE$ of sets, an $\EE$-\emph{rainbow set} is a set $R \subseteq \union\EE$ with an injection $\sigma\colon R \to \EE$ such that $e \in \sigma(e)$ for all $e \in R$. The term rainbow set originates in viewing every member of $\EE$ as a color, and every $e \in R$ as colored by $\sigma(e)$. When we speak of a rainbow set, we often keep in mind the injection $\sigma$, and we say that $\sigma(e) \in \EE$ is \emph{represented} by $e$ in $R$.

\begin{remark}
  Throughout we use the term ``family'' in the sense of ``multiset'' allowing repeated members.
\end{remark}

A recurring theme in the study of rainbow sets is finding an $\EE$-rainbow set satisfying a property $\PP$, assuming that every member of $\EE$ satisfies $\PP$, and that $\EE$ is large. A classic result of this type is  B\'ar\'any's colorful Carath\'eodory theorem~\cite{B82}: every family of $n+1$ subsets of $\R^n$, each containing a point $a$ in its convex hull, has a rainbow set satisfying the same property. An application mentioned in \cite{B82} is a theorem due to Frank and Lov\'asz, on rainbow directed cycles. Other results of this type are about rainbow matchings. For example, improving a theorem of Drisko~\cite{D98}, Aharoni and Berger~\cite[Theorem~4.1]{AB} proved that $2n-1$ matchings of size $n$ in any bipartite graph have a rainbow matching of size $n$. In \cite{AKZ} the examples showing sharpness of this result were characterized, and in \cite{ARJ} the theorem was given a topological proof. A more general context is that of independent sets in graphs, see, e.g., \cite{ABKK,KL,KKK}.

In this paper we study conditions for the existence of rainbow cycles, with or without a parity constraint on their lengths. Hereafter a cycle is viewed as a set of edges. Our main result is:
\begin{theorem}\label{thm:odd-cycles}
  Every family of $2\ceil{n/2}-1$ odd cycles in the complete graph $K_n$ on $n$ vertices has a rainbow odd cycle.
\end{theorem}

Put more explicitly, the theorem states that when $n$ is odd, every family of $n$ odd cycles in $K_n$ has a rainbow odd cycle; when $n$ is even $n-1$ odd cycles suffice. The case of $n$ odd is relatively easy, and the main effort goes into the even case. The proof is done in \cref{sec:roc} via a characterization of families of $n - 1$ odd cycles in $K_n$ without any rainbow odd cycle. In particular, when $n$ is even, $n-1$ odd cycles in $K_n$ cannot form the characterized family.

In \cref{sec:rgc} we deal with rainbow cycles of general length. The fact that $n$ cycles in $K_n$ have a rainbow cycle is easy, and the main result is a characterization of families of $n$ cycles in $K_{n+1}$ without any rainbow cycle. In \cref{sec:rainbow-cycles-in-multigraphs}, we consider rainbow cycles in edge-disjoint families; our result in this case turns out to be a rediscovery, with a short proof, of a theorem of \cite{HHJO}. In \cref{sec:remarks} we conclude with a generalization to matroids, and a result on rainbow even cycles.

\section{Rainbow odd cycles}\label{sec:roc}

We start with an observation which yields \cref{thm:odd-cycles} in the case of $n$ odd.

\begin{proposition}\label{thm:woc}
  Every family of $n$ odd cycles in $K_n$ has a rainbow odd cycle.\footnote{A reworded version of \cref{thm:woc}, suggested by the first author, appeared as Problem 3 of Day 1 in the 12th Romanian Master in Mathematics, RMM 2020.}
\end{proposition}

\begin{proof}
  Let $R$ be a maximal rainbow forest. Since $R$ has fewer than $n$ edges, one of the odd cycles, say $O$, is not represented in $R$. By the maximality of $R$, no edge in $O$ connects two components of $R$. Thus $O$ is contained in a connected component $T$ of $R$. Since $O$ is of odd length, one of its edges does not obey the bipartition of $T$. Adding that edge to $T$ yields a rainbow subgraph that supports an odd cycle.
\end{proof}

A Hamiltonian cycle on $n$ vertices repeated $n-1$ times shows the sharpness of \cref{thm:woc} only for $n$ odd. This example can be generalized as follows.

\begin{definition}
  A family $\OO$ of cycles is a \emph{pruned cactus} if all the cycles in $\OO$ are identical to a fixed cycle on $\abs{\OO} + 1$ vertices, or $\OO$ can be partitioned into two pruned cacti $\OO_1, \OO_2$ such that $\union\OO_1$ and $\union\OO_2$ share exactly one vertex.\footnote{A cactus graph is a connected graph in which two cycles have at most one vertex in common. A pruned cactus $\OO$ is named after the fact that $\union\OO$ is a $2$-edge-connected cactus graph.}
\end{definition}

\begin{figure}
  \centering
  \begin{tikzpicture}[scale=0.02, very thick]
    \coordinate (a1) at (135,-342);
    \coordinate (a2) at (162,-324);
    \coordinate (a5) at (153,-297);
    \coordinate (a6) at (153,-234);
    \coordinate (a9) at (108,-315);
    \coordinate (a10) at (126,-279);
    \coordinate (a11) at (108,-234);
    \coordinate (a15) at (126,-198);
    \coordinate (a16) at (189,-171);
    \coordinate (a17) at (207,-135);
    \coordinate (a20) at (180,-135);
    \coordinate (a21) at (189,-108);
    \coordinate (a22) at (180,-81 );
    \coordinate (a23) at (207,-72 );
    \coordinate (a24) at (198,-54 );
    \coordinate (a28) at (99 ,-162);
    \coordinate (a29) at (135,-162);
    \coordinate (a30) at (117,-126);
    \coordinate (a31) at (135,-90 );
    \coordinate (a32) at (126,-54 );
    \coordinate (a33) at (135,-27 );
    \coordinate (a35) at (117,-9  );
    \coordinate (a36) at (99 ,-54 );
    \coordinate (a39) at (45 ,-90 );
    \coordinate (a40) at (27 ,-117);
    \coordinate (a41) at (63 ,-108);
    \coordinate (a42) at (45 ,-144);
    \coordinate (a43) at (36 ,-180);

    \draw[green] (a1) -- (a2) -- (a5) -- (a6) -- (a15) -- (a11) -- (a10) -- (a9) -- cycle;
    \draw[green] (a15) -- (a29) -- (a30) -- (a28) -- cycle;
    \draw[green] (a30) -- (a31) -- (a32) -- (a33) -- (a35) -- (a36) -- cycle;
    \draw[green] (a15) -- (a16) -- (a20) -- cycle;
    \draw[green] (a20) -- (a17) -- (a21) -- cycle;
    \draw[green] (a21) -- (a23) -- (a24) -- (a22) -- cycle;
    \draw[green] (a28) -- (a42) -- (a43) -- cycle;
    \draw[green] (a42) -- (a41) -- (a39) -- (a40) -- cycle;
    \foreach \i in {1,2,5,6,9,10,11,15,16,17,20,21,22,23,24,28,29,30,31,32,33,35,36,39,40,41,42,43}
      \node[vertex] at (a\i) {};
  \end{tikzpicture}\qquad\qquad%
  \begin{tikzpicture}[scale=0.02, very thick, xscale=-1]
    \coordinate (a1) at (135,-342);
    \coordinate (a2) at (162,-324);
    \coordinate (a5) at (153,-297);
    \coordinate (a6) at (153,-234);
    \coordinate (a9) at (108,-315);
    \coordinate (a10) at (126,-279);
    \coordinate (a11) at (108,-234);
    \coordinate (a15) at (126,-198);
    \coordinate (a16) at (189,-171);
    \coordinate (a17) at (207,-108);
    \coordinate (a20) at (180,-135);
    \coordinate (a21) at (20,-144);
    \coordinate (a22) at (180,-81 );
    \coordinate (a23) at (207,-72 );
    \coordinate (a24) at (198,-54 );
    \coordinate (a28) at (99 ,-162);
    \coordinate (a29) at (135,-162);
    \coordinate (a30) at (117,-126);
    \coordinate (a31) at (135,-90 );
    \coordinate (a32) at (126,-54 );
    \coordinate (a33) at (135,-27 );
    \coordinate (a35) at (117,-9  );
    \coordinate (a36) at (99 ,-54 );
    \coordinate (a39) at (45 ,-90 );
    \coordinate (a40) at (27 ,-117);
    \coordinate (a41) at (63 ,-108);
    \coordinate (a42) at (45 ,-144);
    \coordinate (a43) at (36 ,-180);

    \draw[green] (a1) -- (a5) -- (a6) -- (a15) -- (a11) -- (a10) -- (a9) -- cycle;
    \draw[green] (a15) -- (a29) -- (a31) -- (a32) -- (a33) -- (a35) -- (a36) -- (a30) -- (a28) -- cycle;
    \draw[green] (a15) -- (a16) -- (a20) -- cycle;
    \draw[green] (a20) -- (a17) -- (a23) -- (a24) -- (a22) -- cycle;
    \draw[green] (a28) -- (a42) -- (a43) -- cycle;
    \draw[green] (a42) -- (a41) -- (a39) -- (a40) -- (a21) -- cycle;
    \foreach \i in {1,5,6,9,10,11,15,16,17,20,21,22,23,24,28,29,30,31,32,33,35,36,39,40,41,42,43}
      \node[vertex] at (a\i) {};
  \end{tikzpicture}
  \caption{Underlying graphs of two pruned cacti. The one on the right is composed of odd cycles.}\label{fig:dancers}
\end{figure}
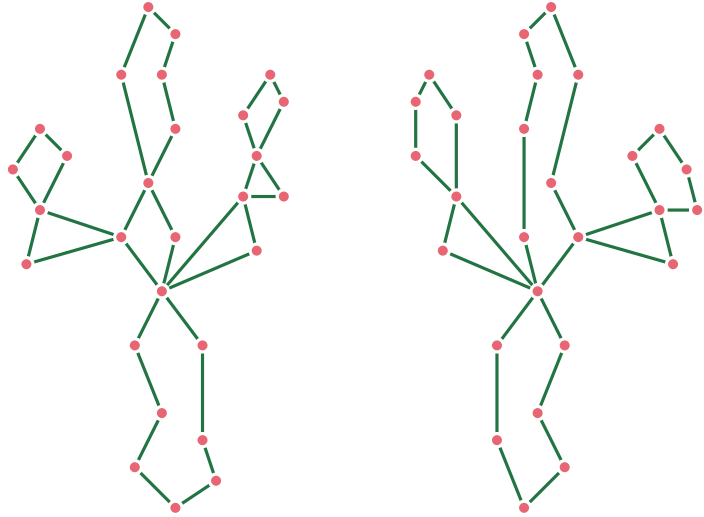

Given a pruned cactus $\OO$, by our recursive definition, one can check that $\OO$ has no rainbow cycle, and the underlying graph $\union\OO$ contains exactly $\abs{\OO} + 1$ vertices (see \cref{fig:dancers}). A key result towards the proof of \cref{thm:odd-cycles} is that the converse is also true for $\OO$ composed of only odd cycles. For technical reasons, we shall switch from now on to cycles in $K_{n+1}$ rather than $K_n$.

\begin{theorem}\label{thm:odd-cycles-char}
  If a family of $n$ odd cycles in $K_{n+1}$ has no rainbow odd cycle, then it is a pruned cactus.
\end{theorem}

Clearly, the cardinality of a pruned cactus composed solely of odd cycles is even. Therefore, when $n$ is even, $n-1$ odd cycles cannot form a pruned cactus, and so \cref{thm:odd-cycles} follows from \cref{thm:odd-cycles-char}.

For the inductive proof of \cref{thm:odd-cycles-char}, we need the following technical lemma.

\begin{lemma}\label{lem:occ}
  Let $\OO := \set{O_1, \dots, O_n}$ be a family of odd cycles in $K_{n+1}$ without any rainbow odd cycle, and denote $\KK := \set{O_1, \dots, O_k}$, where $k < n$. Suppose that $Q$ is a $(k+1)$-vertex subgraph of $\union\KK$ and $V \subseteq V(Q)$ such that every pair of vertices in $V$ can be connected by a $\KK$-rainbow even path in $Q$. Then
  \begin{enumerate}[nosep, label=(\alph*)]
    \item No edge in $O_{k+1}, \dots, O_n$ has both endpoints in $V$.\label{lem:occ-a}
  \end{enumerate}
  Moreover, let $\pi$ be the contraction\footnote{A contraction operation removes all edges between any pair of  contracted vertices.} that replaces $V(Q)$ with a single vertex $\bar{v}$, and suppose that $P_{k+1}, \dots, P_n$ are respectively subgraphs of $O_{k+1}, \dots, O_n$ such that each $P_i$ avoids the vertices in $V(Q) \setminus V$. Denote $\PPP := \set{\pi(P_{k+1}), \dots, \pi(P_n)}$. Then the following holds.
  \begin{enumerate}[nosep, label=(\alph*), resume]
    \item There is no $\PPP$-rainbow odd cycle in $\pi(K_{n+1})$.\label{lem:occ-b}
    \item If $\PPP$ is a pruned cactus of odd cycles, then $\union\PPP$ is spanning in $\pi(K_{n+1})$, and no $O_i \setminus P_i$ contains an edge of the form $uv$ with $u \not\in V(Q) \cup V(P_i)$ and $v \in V \cap V(P_i)$.\label{lem:occ-c}
  \end{enumerate}
\end{lemma}

\begin{proof}
  Note that any edge in $O_{k+1}, \dots, O_{n}$ with both endpoints in $V$ can be completed to an $\OO$-rainbow odd cycle by a $\KK$-rainbow even path in $Q$.

  Assume for the sake of contradiction that there is a $\PPP$-rainbow odd cycle $C$ in $\pi(K_{n+1})$. Edges of the form $u\bar{v}$ after the contraction correspond to edges of the form $uv$ with $v \in V$ before the contraction. Hence, prior to the contraction, $C$ was either itself an $(\OO\setminus \KK)$-rainbow odd cycle (which does not exist), or an ($\OO\setminus \KK$)-rainbow odd path between a pair of vertices in $V$, which can be completed to an $\OO$-rainbow odd cycle by a $\KK$-rainbow even path in $Q$.

  To prove \ref{lem:occ-c}, suppose that the family $\PPP$ is a pruned cactus of odd cycles. Notice that $\pi(K_{n+1})$ has $n + 1 - \abs{V(Q)} + 1 = n - k + 1$ vertices, and the underlying graph $\union\PPP$ of the pruned cactus $\PPP$ has $\abs{\PPP} + 1 = n - k + 1$ vertices. Thus $\union\PPP$ is spanning in $\pi(K_{n+1})$, and so $\bar{v}$ is on $\union\PPP$. Finally, suppose on the contrary that some $O_i \setminus P_i$ contains an edge $uv$ with $u \not\in V(Q) \cup V(P_i)$ and $v \in V \cap V(P_i)$. Since $\bar{v} = \pi(v)$ is on $\pi(P_i)$ and $u$ is not on $\pi(P_i)$, one can find a $\PPP$-rainbow even path from $\bar{v}$ to $u$, in which $\pi(P_i)$ is not represented. This $\PPP$-rainbow even path can then be completed by the edge $\pi(uv) = u\bar{v}$ to a $\set{\pi(P_{k+1}), \dots, \pi(P_{i-1}), \pi(uv), \pi(P_{i+1}),\dots, \pi(P_n)}$-rainbow odd cycle. However this contradicts \ref{lem:occ-b} with $uv$, which is an edge of $O_i$ that avoids $V(Q)\setminus V$, playing the role of $P_i$.
\end{proof}

The last ingredient is a corollary of Rado's theorem for matroids~\cite{R42}, that gives a necessary and sufficient condition for a family of connected subgraphs to have a rainbow spanning tree.

\begin{theorem}[Rado's theorem for matroids]
  Given a matroid with ground set $E$, for every family $\set{E_1, \dots, E_m}$ of subsets of $E$, there exists a rainbow independent set of size $m$ if and only if $\rank{E_I} \ge \abs{I}$ for every $I \subseteq [m]$, where $E_I$ is shorthand for $\bigcup_{i \in I}E_i$.
\end{theorem}

\begin{corollary}\label{lem:rainbow-spanning-tree}
  For every family $\set{E_1, \dots, E_m}$ of connected subgraphs (viewed as edge sets) in $K_{m+1}$, the family has a rainbow spanning tree if and only if $\abs{V(E_I)} \ge \abs{I} + 1$ for every $I \subseteq [m]$.
\end{corollary}

\begin{proof}
  The ``only if'' direction is easy to check. For the ``if'' direction, it suffices to verify the rank inequalities in Rado's theorem for matroids. Recall that, in a graphic matroid, $\rank{E} = \abs{V(E)} - c(E)$ for every edge set $E$, where $c(E)$ is the number of connected components of $E$. Pick an arbitrary $I \subseteq [m]$. Because each $E_i$ is connected, we can partition $I$ into sets $I_1, \dots, I_c$, where $c := c(E_I)$, such that $E_{I_1}, \dots, E_{I_c}$ are the connected components of $E_I$. Since $\abs{V(E_{I_j})} \ge \abs{I_j} + 1$ for all $j \in [c]$, we have the desired inequality
  \begin{equation*}
    \rank{E_I} = \abs{V(E_I)} - c(E_I) = \sum_{j = 1}^c \left(\abs{V(E_{I_j})} - 1\right) \ge \sum_{j = 1}^c \abs{I_j} = \abs{I}.\qedhere
  \end{equation*}
\end{proof}

\begin{proof}[Proof of \cref{thm:odd-cycles-char}]
  We do this by induction. The base case $n=2$ is trivial. Suppose $n \ge 3$, and let $\OO=\set{O_1, \dots, O_n}$ be a family of odd cycles in $K_{n+1}$ without any rainbow odd cycle. We break the inductive step into three cases.

  \paragraph{Case 1:}
  There exists a proper subfamily $\KK$ of $\mathcal{O}$ such that $\abs{V(\union\KK)} \le \abs{\KK}+1$.

  Since there is no $\KK$-rainbow odd cycle, by the induction hypothesis $\KK$ is a pruned cactus. By passing to a subfamily of $\KK$, we may assume without loss of generality that $\KK=\set{O_1, \dots, O_k}$, for some $k < n$, and $O_1, \dots, O_k$ are identical to a fixed odd cycle $O$ on $k+1$ vertices. Note that every pair of vertices in $V(O)$ can be connected by a $\KK$-rainbow even path in $O$. By \cref{lem:occ}\ref{lem:occ-a}, for every $i \in \set{k+1, \dots, n}$,  the arcs of $O_i$ defined by its vertices shared with $O$ are of length $\ge 2$. Since $O_i$ is odd, there exists an odd arc, call it $P_i$. In case $O_i$ and $O$ are vertex-disjoint, set $P_i := O_i$.

  Let $\pi$ be the contraction of $V(O)$ to a single vertex $\bar{v}$. By our choice of $P_i$, for each $i > k$, $\pi(P_i)$ is an odd cycle, and so \cref{lem:occ}\ref{lem:occ-b} and the inductive hypothesis imply that the family $\PPP := \set{\pi(P_{k+1}), \dots, \pi(P_n)}$ is a pruned cactus.

  \begin{claim*}
    For every $i > k$, $P_i = O_i$, in other words, $O_i$ and $O$ share at most $1$ vertex.
  \end{claim*}

  Assume for contradiction that $P_i \neq  O_i$ for some $i > k$. Let $uv$ be an edge in $O_i \setminus P_i$ with $u \not\in V(P_i)$ and $v \in V(P_i)$. Note in addition that $u \not\in V(O)$ by \cref{lem:occ}\ref{lem:occ-a}, while $v \in V(O)$, which conflicts with \cref{lem:occ}\ref{lem:occ-c}.

  \begin{claim*}
    For every $i, j > k$, if $\pi(O_i)=\pi(O_j)$, then $O_i=O_j$.
  \end{claim*}

  Suppose on the contrary that $\pi(O_i)=\pi(O_j)$ and $O_i \neq O_j$ for some $i,j > k$. Let $v_i,v_j$ be respectively the vertices of $O_i, O_j$ shared with $O$. Then there exists an $\set{O_i,O_j}$-rainbow cherry with endpoints $v_i,v_j$ and a center not in $V(O)$, which can be completed to an $\OO$-rainbow odd cycle by a $\KK$-rainbow odd path in $O$.

  By \cref{lem:occ}\ref{lem:occ-c}, $\bar{v}\in V(\union\PPP)$, implying that $\union\OO$ is connected. By the last claim, for $i > k$, the multiplicity of every $O_i$ in $\OO$ is equal to the multiplicity of $\pi(O_i)$ in $\PPP$, which, by the fact that $\PPP$ is a pruned cactus, is $\abs{O_i}-1$. Together, this means that $\OO$ is a pruned cactus, as desired.

  \paragraph{Case 2:} Every odd cycle $O_i$ is Hamiltonian.

  Let $S$ be an $\OO$-rainbow star of maximum size, say $k$, and let $c$ be its center.\footnote{A \emph{star} of size $k$ is a set of $k \ge 2$ edges, sharing one vertex that is called the \emph{center} of the star.} Without loss of generality, we may assume that the cycles represented in $S$ are $O_1, \dots, O_k$. We may further assume that the cycles in $\OO$ are not identical for otherwise $\OO$ is already a pruned cactus.

  \begin{claim*}
    The size $k$ of $S$ satisfies $3 \le k <n$.
  \end{claim*}

  Because the cycles in $\OO$ are not identical, there is a vertex $v$ in $\union\OO$ of degree at least $3$. A quick argument shows an $\OO$-rainbow star of size $3$ centered at $v$, meaning that $k \ge 3$. Negation of the second inequality means that $c$ is connected in $S$ to all other vertices of the graph. Suppose $O_1$ is represented by $cv$ in $S$. In the absence of an $\OO$-rainbow triangle, no edge of $O_1$ has both endpoints in $V(K_{n+1})\setminus\set{c, v}$. Because $\abs{V(K_{n+1})\setminus\set{c, v}} = n - 1 \ge 2$, it is impossible for $O_1$ to be Hamiltonian given that $cv$ is already in $O_1$.

  Let $V$ be the set of leaves of $S$. Since $\OO$ has no rainbow triangle, the cycles $O_{k+1}, \dots, O_n$ do not connect pairs of vertices of $V$. By the maximality of $S$, these cycles enter and exit $c$ through $V$. Therefore, for every $i > k$, $V$ partitions $O_i$ into arcs of length at least two, and at least one of these arcs, call it $P_i$, is odd and does not contain $c$.

  Let $\pi$ be the contraction that replaces $V(S)$ by a single vertex $\bar{v}$. As in Case~1, the family $\{\pi(P_{k+1}), \dots, \pi(P_n)\}$ is a pruned cactus of odd cycles.

  Since $k \ge 3$, $V$ partitions $O_n$ into at least $3$ arcs, one of which is next to $P_n$ and does not contain $c$. Hence $O_n \setminus P_n$ contains an edge $uv$ with $u \not\in V(S) \cup V(P_n)$ and $v \in V \cap V(P_n)$, which contradicts \cref{lem:occ}\ref{lem:occ-c}.

  \paragraph{Case 3:} For every proper subfamily $\KK$ of $\OO$, $\abs{V(\union\KK)} > \abs{\KK} + 1$, and some $O_i$ is not Hamiltonian.

  Without loss of generality, assume that $O_n$ does not contain some vertex $v$. Set $V := V(K_{n+1}) \setminus \set{v}$. We apply \cref{lem:rainbow-spanning-tree} to the family of subgraphs $O_1[V], \dots, O_{n-1}[V]$ induced by $V$, and obtain an $\set{O_1, \dots, O_{n-1}}$-rainbow tree $T$ that spans $V$. Since $O_n$ is of odd length, one of its edges does not obey the bipartition of $T$. Adding that edge to $T$ yields a rainbow subgraph that supports an odd cycle.
\end{proof}

\section{Rainbow cycles}\label{sec:rgc}

Here is a cheap bound on the size of the family that ensures a rainbow set with a certain property.

\begin{proposition}\label{lem:naive}
  Given a ground set $E$ and a property $\PP \subseteq 2^E$ with $\varnothing\not\in \PP$ that is closed upwards, every family of $m+1$ subsets $E_1, \dots, E_{m+1}$ of $E$ with each $E_i \in \PP$ has a rainbow set in $\PP$, where
  \[
    m := \max\dset{\abs{F}}{F \subseteq E \text{ and }F \not\in \PP}.
  \]
\end{proposition}

\begin{proof}
  Take $R$ to be a rainbow subset of $E$ not in $\PP$ of maximum size. Since $R \not\in \PP$, $\abs{R} \le m$ and some $E_i$ is not represented in $R$. Because $E_i \in \PP$, $E_i \neq \varnothing$, and moreover because $\PP$ is closed upwards, $E_i \not\subseteq R$. Take $e \in E_i \setminus R$ and define $R' := R \cup \set{e}$, which is rainbow. By the maximality of $R$, we know that $R' \in \PP$.
\end{proof}

For rainbow cycles, simply note that a subgraph of $K_n$ without cycles, that is a forest, contains at most $n-1$ edges.

\begin{proposition}\label{lem:rgc}
  Every family of $n$ cycles in $K_n$ has a rainbow cycle. \qed
\end{proposition}

The sharpness of \cref{lem:rgc} is witnessed by a pruned cactus. But there is a more general construction showing this.

\begin{definition}
  A family $\OO$ of cycles is a \emph{saguaro} if the family $\OO$ is already a pruned cactus, or the family $\OO$ can be partitioned into three subfamilies $\OO_1, \set{O}, \OO_2$ such that $\OO_1$ and $\OO_2$ are two vertex-disjoint saguaros, and $O$ is an even cycle along which its vertices alternate between $V(\union \OO_1)$ and $V(\union \OO_2)$.
\end{definition}

One can inductively check that if $\OO$ is a saguaro then $\OO$ has no rainbow cycle, and $\abs{V(\union\OO)} = \abs{\OO} + 1$. We prove that this recursive construction is an exhaustive characterization of families of $n$ cycles in $K_{n+1}$ without any rainbow cycle.

\begin{theorem}\label{thm:cycles-char}
  For every family $\OO$ of $n$ cycles in $K_{n+1}$, no rainbow cycle exists if and only if the family is a saguaro.
\end{theorem}

Our proof strategy parallels the proof of \cref{thm:odd-cycles-char}, with  a few detours. A complication arises when an even cycle, after contracting its maximum independent set, becomes a star. To handle this problem, we shall use the following:

\begin{proposition}\label{lem:almost-spanning-tree}
  Let $v$ be a vertex of $K_{m+1}$, and let $\EE := \set{E_1, \dots, E_m}$ be a family of subgraphs of $K_{m+1}$, where each $E_i$ is either a star centered at $v$ or a cycle. Suppose that $\EE$ has no rainbow cycle, and every star in $\EE$ is edge-disjoint from all the other members  of $\EE$. If $E_1$ is a star, then there are $\ell$ cycles in $\EE$ avoiding $v$, for some $0 < \ell < m$, whose union with $E_1$ contains at most $\ell+2$ vertices.
\end{proposition}

\begin{proof}
  Let $R$ be a maximal $\set{E_2, \dots, E_m}$-rainbow tree containing $v$. We may assume that such a tree exists, since otherwise $E_2, \dots, E_m$ are cycles as required.

  Without loss of generality, assume that $E_2, \dots, E_k$ are represented in $R$, where $k = \abs{V(R)}$. Since $E_1$ is edge-disjoint from $E_i$ for $i \neq 1$, it is edge-disjoint from $R$. Furthermore, since $\EE$ has no rainbow cycle, $R$ does not contain any leaf of $E_1$. Since a star has at least two edges, it follows that  $k \le m - 1$.

  \begin{claim*}
    For every $i>k$, $E_i$ is a cycle that is vertex-disjoint from $R$.
  \end{claim*}

  The fact that $E_i$ is a cycle follows from the maximality of $R$ and the requirement that every star in $\EE$ is edge-disjoint from all other members of $\EE$. The disjointness from $R$ follows from the assumption that $\EE$ has no rainbow cycle.

  Let $\ell = m - k$. By the claim, $E_{k+1}, \dots, E_m$ are the desired $\ell$ cycles since their vertex sets, as well as that of $E_1$,  are contained in  $(V(K_{m+1}) \setminus V(R)) \cup \set{v}$, which is of size $m + 1 - k + 1 = \ell + 2$.
\end{proof}

Unlike in a pruned cactus, not every cycle in a saguaro is repeated more than once. We say an $\ell$-cycle is \emph{common} in the family if it is repeated exactly $\ell-1$ times. We shall use the following technical lemma that is analogous to \cref{lem:occ}.

\begin{lemma}\label{lem:gcc}
  Let $\OO := \set{O_1, \dots, O_n}$ be a family of cycles in $K_{n+1}$ without any rainbow cycle, and denote $\KK := \set{O_1, \dots, O_k}$, where $k < n$. Suppose that $Q$ is a $(k+1)$-vertex subgraph of $\union\KK$, and $V \subseteq V(Q)$ such that every pair of vertices in $V$ can be connected by a $\KK$-rainbow path of length at least $2$ in $Q$. Then
  \begin{enumerate}[nosep, label=(\alph*)]
    \item No edge in $O_{k+1}, \dots, O_n$ has both endpoints in $V$.\label{lem:gcc-a}
  \end{enumerate}
  Moreover, let $\pi$ be the contraction that replaces $V(Q)$ by a single vertex $\bar{v}$, and suppose $P_{k+1}, \dots, P_n$ are respectively subgraphs of $O_{k+1}, \dots, O_n$ such that each $P_i$ avoids the vertices in $V(Q) \setminus V$. Denote $\PPP := \set{\pi(P_{k+1}), \dots, \pi(P_n)}$. Then the following holds.
  \begin{enumerate}[nosep, label=(\alph*), resume]
    \item There is no $\PPP$-rainbow cycle in $\pi(K_{n+1})$.\label{lem:gcc-b}
    \item If $\PPP$ is a saguaro of cycles, then $\union \PPP$ is spanning in $\pi(K_{n+1})$. Moreover, for every $\pi(P_i)$ that is common in $\PPP$, $O_i \setminus P_i$ does not contain any edge of the form $uv$ with $u \not\in V(Q) \cup V(P_i)$ and $v \in V \cap V(P_i)$.\label{lem:gcc-c}
  \end{enumerate}
\end{lemma}

We leave the proof to the readers as it is similar to that of \cref{lem:occ}.

\begin{proof}[Proof of \cref{thm:cycles-char}]
  The ``if'' direction is easy to check. We show the ``only if'' direction by induction. The base case $n=2$ is trivial. Suppose $n \ge 3$, and $\OO := \set{O_1, \dots, O_n}$ is a family of cycles in $K_{n+1}$ without any rainbow cycle. We break the inductive step into three cases.

  \paragraph{Case 1:} There exists a proper subfamily $\KK$ of $\OO$ such that $\abs{V(\union\KK)} \le \abs{\KK}+1$.

  Let $\KK$ be maximal with this property. Without loss of generality, $\KK = \set{O_1, \dots, O_k}$, where $k := \abs{\KK} < n$. Set $V := V(\union \KK)$. By the induction hypothesis, $\KK$ is a saguaro. In particular, as can be observed in any saguaro, $\abs{V} = k + 1$ and every pair of vertices in $V$ can be connected by a $\KK$-rainbow path of length at least $2$. For every $i > k$, by \cref{lem:gcc}\ref{lem:gcc-a}, the arcs of $O_i$ defined by its vertices on $V$ are of length at least $2$. If there exists an arc of length $\ge 3$, choose one such arc and denote it by $P_i$. If there is no such arc, set $P_i := O_i$. In case $O_i$ avoids $V$, also set $P_i := O_i$.

  Let $\pi$ be the contraction that replaces $V$ by a single vertex $\bar{v}$. Then $\pi(P_i)$ is a cycle, with one possible exception: the vertices of $O_i$ alternate between $V$ and $V(K_{n+1})\setminus V$. In the latter case, $P_i = O_i$ and $\pi(P_i)$ is a star centered at $\bar{v}$ (with at least $2$ edges).

  We next break the current case into two subcases.

  \medskip
  \noindent\textbf{Subcase 1.1:} For every $i > k$, $\pi(P_i)$ is a cycle.

  \cref{lem:gcc}\ref{lem:gcc-b} and the inductive hypothesis imply that the family $\PPP := \set{\pi(P_{k+1}), \dots, \pi(P_n)}$ is a saguaro. By \cref{lem:gcc}\ref{lem:gcc-c}, $\bar{v} \in V(\union\PPP)$. As can be observed in any saguaro, there is a common cycle in $\PPP$ that contains $\bar{v}$. Let this cycle have length $\ell+1$, and assume without loss of generality that it appears in $\PPP$ as $\pi(P_{k+1}), \dots, \pi(P_{k+\ell})$.

  \begin{claim*}
    For every $i \in \set{k+1, \dots, k+\ell}$, $P_i = O_i$.
  \end{claim*}

  Suppose on the contrary that $P_i \neq O_i$ for some $i \in \set{k+1, \dots, k+\ell}$. Then one of the two edges in $O_i$, say $uv$, adjacent to $P_i$, satisfies $u \not\in V$ and $v \in V(P_i) \cap V$, contradicting \cref{lem:gcc}\ref{lem:gcc-c}.

  Since $\pi(O_{k+1}), \dots, \pi(O_{k+\ell})$ are the same cycle of length $\ell + 1$, the union of $O_1, \dots, O_{k+\ell}$ contains $k + \ell + 1$ vertices. By the maximality property of $\KK$, it therefore follows $k + \ell = n$, in other words, $\pi(O_{k+1}), \dots, \pi(O_n)$ are the same cycle.

  \begin{claim*}
    The cycles $O_{k+1}, \dots, O_n$ also coincide.
  \end{claim*}

  The reason is that if $O_i\neq O_j$ for some $i,j >k$, then there exists an $\set{O_i, O_j}$-rainbow cherry with endpoints in $V$, that can be completed to an $\OO$-rainbow cycle by a $\KK$-rainbow path.

  As in the parallel stage of the proof of \cref{thm:odd-cycles-char}, the last claim implies that $\OO$ is  a saguaro.

  \medskip
  \noindent\textbf{Subcase 1.2:} For some $i > k$, $\pi(P_i)$ is a star centered at $\bar{v}$.

  Without loss of generality $\pi(P_{k+1})$ is a star centered at $\bar{v}$. Recall that each member in $\PPP$ is either a star centered at $\bar{v}$ or a cycle. Moreover \cref{lem:gcc}\ref{lem:gcc-b} implies that $\PPP$ has no rainbow cycle.

  \begin{claim*}
    Every star in $\PPP$ is edge-disjoint from all the other members of $\PPP$.
  \end{claim*}

  Indeed, assume that for some $i, j > k$ we have an edge $u\bar{v}$ shared by $\pi(P_i)$ and $\pi(P_j)$, where $\pi(P_i)$ is a star centered at $\bar{v}$. Then in $O_i$ the vertex $u$ has two neighbors in $V$ and in $O_j$ it has at least one neighbor in $V$. Hence there is an $\set{O_i,O_j}$-rainbow cherry with endpoints in $V$ and center $u$, which can be completed to an $\OO$-rainbow cycle by a $\KK$-rainbow path.

  By \cref{lem:almost-spanning-tree} it follows that there exist $\ell$ cycles in $\PPP$ avoiding $\bar{v}$, say $\pi(P_{k+2}), \dots, \pi(P_{k+\ell+1})$, whose union with $\pi(P_{k+1})$ contains at most $\ell + 2$ vertices, one of them being $\bar{v}$. Note that if $\pi(P_i)$ is a cycle avoiding $\bar{v}$, then $P_i = O_i$. Hence the union of $O_1, \dots, O_{k+\ell+1}$ contains at most $(k + 1) + (\ell + 1)$ vertices. To reconcile this with our choice of $\KK$, the only way out is that $k + \ell + 1 = n$ and none of $\pi(P_{k+2}), \dots, \pi(P_n)$ contains $\bar{v}$. Thus all of $O_{k+2}, \dots, O_n$ avoid $V$, and so the union of these $n - k - 1$ cycles contains at most $n - k$ vertices. By the induction hypothesis, the subfamily $\set{O_{k+2}, \dots, O_n}$ is a saguaro of cycles that avoids $V$. Recall that the vertices of $O_{k+1}$ alternate between $V$ and $V(K_{n+1})\setminus V$. Therefore $\OO$ is a saguaro.

  \paragraph{Case 2:} Every cycle $O_i$ is Hamiltonian.

  Let $S$ be an $\OO$-rainbow star of maximum size, say $k$. Without loss of generality, assume that the cycles represented in $S$ are $O_1, \dots, O_k$. Denote $\KK := \set{O_1, \dots, O_k}$. As in the proof of \cref{thm:odd-cycles-char}, using the fact that $\OO$ has no rainbow cycle, we can deduce that $k < n$. As there, if $k = 2$ then all the cycles in $\OO$ are identical, so we may assume $k \ge 3$.

  Let $c$ be the center of $S$ and $V$ the set of its leaves. Notice that every pair of vertices in $V$ can be connected by a $\KK$-rainbow cherry. For an arbitrary $i > k$, by \cref{lem:gcc}\ref{lem:gcc-a}, $V$ is an independent set of $O_i$, and so $k \le (n+1)/2$.

  Suppose for a moment that $k = (n+1)/2$. Since $n - k = k - 1 \ge 2$, there are at least $2$ cycles in $\OO\setminus\KK$, and there is a vertex $u \not\in V(S)$. Note that $V$ partitions both $O_{k+1}$ and $O_{k+2}$ into arcs of length $2$. The two arcs through $u$ obtained respectively from $O_{k+1}$ and $O_{k+2}$ yield an $\set{O_{k+1},O_{k+2}}$-rainbow cherry with endpoints in $V$ and center $u$, which can be completed to an $\OO$-rainbow square by a $\KK$-rainbow cherry in $S$.

  Therefore $k < (n+1)/2$. Now, for every $i > k$, one of the arcs, call it $P_i$, of $O_i$ defined by $V$ is of length at least $3$. By the maximality of $S$, $P_i$ does not contain $c$. Let $\pi$ be the contraction that replaces $V(S)$ by $\bar{v}$. Again the family $\PPP := \set{\pi(P_{k+1}), \dots, \pi(P_n)}$ is a saguaro of cycles. Say $\pi(P_n)$ is a common cycle in $\PPP$. Note that $O_n$ is partitioned into at least $3$ arcs by $V$ because $\abs{V} = k \ge 3$. Thus one of the two edges in $O_n$, say $uv$, adjacent to $P_n$ satisfies $u \not\in V(S) \cup V(P_n)$ and $v \in V \cap V(P_n)$, which contradicts \cref{lem:gcc}\ref{lem:gcc-c}.

  \paragraph{Case 3:} For every proper subfamily $\KK$ of $\OO$, $\abs{V(\union\KK)} > \abs{\KK} + 1$, and some $O_i$ is not Hamiltonian.

  The analysis of the last case can be taken almost verbatim from the proof of \cref{thm:odd-cycles-char}.
\end{proof}

\section{Edge-disjoint families}\label{sec:rainbow-cycles-in-multigraphs}

Here we continue to pursue a rainbow cycle, but make the additional assumption that our family consists of pairwise disjoint sets of edges. In terms of colors, this amounts to the natural restriction that every edge of the underlying graph gets just one color.\footnote{When an edge gets two colors, one may or may not want to consider this a rainbow cycle of length $2$ (a \emph{digon}). In this paper we consider only cycles of length $3$ or more. If digons are allowed, then the restriction to edge-disjoint families serves to avoid this trivial kind of rainbow cycle.}

For a family $\EE$ of $n$ disjoint edge sets in $K_n$, we no longer need to assume that each set in $\EE$ is a cycle in order to guarantee a rainbow cycle. The following trivial observation holds.

\begin{proposition}\label{lem:edge-disjoint}
  Every family of $n$ edge-disjoint nonempty subgraphs of $K_n$ has a rainbow cycle. \qed
\end{proposition}

The sharpness of the above is witnessed by a family of single edges forming a spanning tree. But there is a more general construction showing this.

\begin{definition}
A family $\EE$ of graphs is a \emph{linkleaf} if it is an empty family (which we consider as having a ground set of one vertex), or the family $\EE$ can be partitioned into three subfamilies $\EE_1, \{E\}, \EE_2$ such that $\EE_1$ and $\EE_2$ are two (possibly empty) vertex-disjoint linkleaves, and $E$ is a nonempty bipartite graph with respect to the bipartition $V(\union \EE_1), V(\union \EE_2)$.
\end{definition}

We prove below that this recursive construction is a characterization of families of $n$ edge-disjoint nonempty subgraphs of $K_{n+1}$ without any rainbow cycle.

\begin{theorem}\label{thm:linkleaf}
  For every family $\EE$ of $n$ edge-disjoint nonempty subgraphs of $K_{n+1}$, no rainbow cycle exists if and only if the family is a linkleaf.
\end{theorem}

The main part of the proof consists of the following lemma.

\begin{lemma}\label{lem:cut}
  Let $\EE$ be a family of $n$ edge-disjoint nonempty subgraphs of $K_{n+1}$, where $n \ge 1$. If $\EE$ has no rainbow cycle then $\EE$ has a monochromatic cut, that is, a partition $V(\union \EE) = V_1 \cup V_2$ such that exactly one member of $\EE$ has an edge (or more) from $V_1$ to $V_2$.
\end{lemma}

\begin{proof}[Proof of \cref{thm:linkleaf} assuming \cref{lem:cut}]
  The ``if'' direction can easily be verified from the construction. For the ``only if'' direction we use induction. The base case $n=0$ is trivial. Let $\EE$ be a family of $n \ge 1$ edge-disjoint nonempty subgraphs of $K_{n+1}$ without any rainbow cycle. By \cref{lem:cut}, there exists a partition of $V(K_{n+1})$ into $V_1$, say of size $k+1$, and $V_2$, say of size $\ell + 1$, where $k+\ell = n-1$, and a unique member $E$ of $\EE$ having an edge or more from $V_1$ to $V_2$. Since $\EE$ has no rainbow cycle, by \cref{lem:edge-disjoint}, at most $k$ of the subgraphs have an edge or more in $V_1$ and at most $\ell$ of them have an edge or more in $V_2$. Because the total number of members of $\EE$ is $k+\ell + 1$, exactly $k$ of them are contained in $V_1$, exactly $\ell$ of them are contained in $V_2$, and $E$ has only edges from $V_1$ to $V_2$. It follows from the induction hypothesis that $\EE$ is a linkleaf.
\end{proof}

\begin{proof}[Proof of \cref{lem:cut}]
  Assume for the sake of contradiction that the family $\EE := \set{E_1, \dots, E_n}$ has neither a rainbow cycle nor a monochromatic cut. Pick an arbitrary edge $e_i$ from $E_i$ for each $i$, and let $T$ be the rainbow set $\set{e_1, \dots, e_n}$. Since $T$ contains no cycle, $T$ must be a rainbow spanning tree.

  We form a digraph $D$ with vertex set $[n]$, in which an arrow goes from $i$ to $j$, for $i\neq j$, if some edge of $E_j$ reconnects $T \setminus \set{e_i}$. Due to the nonexistence of monochromatic cuts in the family, for every $i$, some edge in $E_j$, for some $j \neq i$, reconnects $T \setminus \set{e_i}$. Thus the minimum out-degree of $D$ is at least $1$.

  Without loss of generality, let $1\to 2\to \dots \to k\to 1$ be a minimum circuit in $D$. As such, let $f_i$ be an edge in $E_i$ that reconnects $T \setminus \{e_{i-1}\}$, for each $i \in [k]$, under the convention that $e_0 := e_k$. Write $O_i$ for the unique cycle formed by adding $f_i$ to $T$. Certainly $e_{i-1}$ is in $O_i$, and moreover $e_i$ is in $O_i$ as $O_i$ cannot be rainbow. By the minimality of the circuit, for each $i, j \in [k]$ with $i \neq j$ and $i \neq j-1 \pmod k$, we have $i \not\to j$ in $D$, which means that $f_j$ does not reconnect $T \setminus \set{e_i}$, and so $e_i \not\in O_j$.

  To summarize, for each $i, j \in [k]$, $e_i \in O_j$ if and only if $i = j$ or $i = j-1 \pmod k$. Set $O := O_1 \bigtriangleup \dots \bigtriangleup O_k$, where $\bigtriangleup$ stands for symmetric difference. Note that
  \[
    \set{f_1, \dots, f_k} \subseteq O \subseteq (T \setminus \set{e_1, \dots, e_k}) \cup \set{f_1, \dots, f_k}.
  \]
  By the first inclusion $O$ is nonempty, and by the second inclusion, it is rainbow. Since every vertex has even degree in $O$, it contains a rainbow cycle.
\end{proof}

\begin{remark}
  It has come to our attention that \cref{thm:linkleaf} already appeared in \cite{HHJO}, and very recently it was generalized for binary matroids by B\'erczi and Schwarcz~\cite{BS}. We still include our proof as it is  elementary and transparent, and moreover it can be easily adapted for binary matroids. In the adapted proof, the binary-ness of the matroids is only needed in the last step of \cref{lem:cut} to show $O$, a symmetric difference of circuits, is a disjoint union of circuits. The rest of our argument works over arbitrary matroids.
\end{remark}

\section{Concluding remarks}\label{sec:remarks}

\subsection{Rainbow spanning in matroids}

\cref{thm:woc} can be seen as a special case of the following rainbow result for matroids.

\begin{proposition}\label{thm:matroid-spanning}
  Let $M$ be a matroid of rank $n$ and let $e$ be an element in the ground set $E$ of $M$. For every family $\set{A_1, \dots, A_n}$ of subsets of $E$, if each $A_i$ contains $e$ in its closure, then the family has a rainbow set that contains $e$ in its closure.
\end{proposition}

\begin{proof}
  Let $R$ be a maximal rainbow set such that $R \cup \set{e}$ is independent in $M$. If $e \in R$ we are done, so assume $e \not\in R$. As the rank of $M$ is $n$, we know that $\abs{R} < n$, and hence some $A_i$ is not represented in $R$. Denote by $\mathrm{span}(\cdot)$ the closure operator in $M$. Since $e \in \mathrm{span}(A_i) \setminus \mathrm{span}(R)$, there exists $a \in A_i \setminus \mathrm{span}(R)$. The set $R' := R \cup \set{a}$ is then a rainbow independent set, and by the maximality of $R$, we have that $R' \cup \set{e}$ is dependent, which implies $e \in \mathrm{span}(R')$.
\end{proof}

To see that \cref{thm:woc} follows from \cref{thm:matroid-spanning}, note that for every edge set $O$ whose vertex set is contained in $[n]$, $O$ contains an odd cycle if and only if $e_0 \in \mathrm{span}(A)$, where $e_0, e_1, \dots, e_n$ form the standard basis of $\F_2^{n+1}$ and $A := \dset{e_0 + e_i + e_j}{\set{i,j}\in O}$. This observation allows us to go back and forth between odd cycles in $K_n$ and subsets of $E$ that contain $e_0$ in their closures, where
\[
  E := \dset{(x_0, x_1, \dots, x_n) \in \F_2^{n+1}}{x_1 + \dots + x_n = 0}
\]
is the ground set of a binary matroid of rank $n$.

\subsection{Rainbow even cycles}

Perhaps surprisingly, the analog of \cref{thm:woc} and \cref{lem:rgc} for even cycles is false. \cref{fig:no-rainbow-even-cycles} shows a family of $6$ squares ($4$-cycles) on $6$ vertices without a rainbow even cycle. By gluing copies of this construction, so that every new copy shares one vertex with the union of the previous ones, we get a family of roughly $6n/5$ squares on $n$ vertices without a rainbow even cycle.

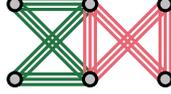
\begin{figure}
  \centering
  \begin{tikzpicture}[scale=1, very thick]
    \coordinate (a1) at (0,0);
    \coordinate (a2) at (0,1);
    \coordinate (a3) at (1,0);
    \coordinate (a4) at (1,1);
    \coordinate (a5) at (2,0);
    \coordinate (a6) at (2,1);

    \draw[green] (a1) -- (a4);
    \draw[green, transform canvas={xshift=1.414pt, yshift=-1.414pt}] (a1) -- (a4);
    \draw[green, transform canvas={xshift=-1.414pt, yshift=1.414pt}] (a1) -- (a4);

    \draw[green] (a2) -- (a3);
    \draw[green, transform canvas={xshift=1.414pt, yshift=1.414pt}] (a2) -- (a3);
    \draw[green, transform canvas={xshift=-1.414pt, yshift=-1.414pt}] (a2) -- (a3);

    \draw[green] (a1) -- (a3);
    \draw[green, transform canvas={yshift=-2pt}] (a1) -- (a3);
    \draw[green, transform canvas={yshift=2pt}] (a1) -- (a3);

    \draw[green] (a2) -- (a4);
    \draw[green, transform canvas={yshift=-2pt}] (a2) -- (a4);
    \draw[green, transform canvas={yshift=2pt}] (a2) -- (a4);

    \draw[red] (a3) -- (a6);
    \draw[red, transform canvas={xshift=1.414pt, yshift=-1.414pt}] (a3) -- (a6);
    \draw[red, transform canvas={xshift=-1.414pt, yshift=1.414pt}] (a3) -- (a6);

    \draw[red] (a4) -- (a5);
    \draw[red, transform canvas={xshift=-1.414pt, yshift=-1.414pt}] (a4) -- (a5);
    \draw[red, transform canvas={xshift=1.414pt, yshift=1.414pt}] (a4) -- (a5);

    \draw[red] (a3) -- (a4);
    \draw[red, transform canvas={xshift=-2pt}] (a3) -- (a4);
    \draw[red, transform canvas={xshift=2pt}] (a3) -- (a4);

    \draw[red] (a5) -- (a6);
    \draw[red, transform canvas={xshift=-2pt}] (a5) -- (a6);
    \draw[red, transform canvas={xshift=2pt}] (a5) -- (a6);

    \foreach \i in {1,2,3,4,5,6}
        \node[vtx] at (a\i) {};
  \end{tikzpicture}
  \caption{A family of $6$ squares on $6$ vertices without any rainbow even cycle.}\label{fig:no-rainbow-even-cycles}
\end{figure}

To get an upper bound on the number of even cycles needed to guarantee a rainbow even cycle, we observe that each connected component of a graph without even cycles is a cactus graph.\footnote{Indeed, if two odd cycles share two vertices, this gives rise to an even cycle.} Note that the densest cactus graph on $n$ vertices is a triangular cactus graph\footnote{A cactus graph is triangular when every cycle in it is a triangle.} (with one bridge if $n$ is even). Thus the maximum number of edges in a graph on $n$ vertices without even cycles is $\lfloor{3(n-1)/2}\rfloor$. From \cref{lem:naive} we have the following rainbow result.

\begin{proposition}\label{lem:even-cycles}
  Every family of $\floor{3(n-1)/2}+1$ even cycles in $K_n$ has a rainbow even cycle.\qed
\end{proposition}

This upper bound is not sharp: for example, $4$ even cycles on $4$ vertices always have a rainbow even cycle.\footnote{The upper bound becomes sharp if the family is allowed to consist of digons (cycles of length $2$). This is seen by taking a triangular cactus graph on $n$ vertices with $\lfloor 3(n-1)/2 \rfloor$ edges, and using one digon for each of its edges. Strictly speaking, however, this takes us from graphs to multigraphs.} We leave the determination of the exact number needed in general (between roughly $6n/5$ and $3n/2$) as an open problem.

\section*{Acknowledgements}

We acknowledge the financial support from the Ministry of Educational and Science of the Russian Federation in the framework of MegaGrant no.\ 075-15-2019-1926 when the first author worked on \cref{sec:roc,sec:rgc} of the paper. We thank the anonymous referees for the suggestions which improved the clarity of the paper.

\bibliographystyle{plain}
\bibliography{cactus}

\end{document}